\documentclass{amsart}
\usepackage{amscd,amssymb,subfigure,hyperref, epsfig}
\usepackage[arrow,matrix,graph,frame,poly,arc,tips]{xy}
\usepackage{amsmath}

\theoremstyle{plain}
\newtheorem{thm}[subsection]{Theorem}
\newtheorem{defn}[subsection]{Definition}
\newtheorem{lem}[subsection]{Lemma}

\newtheorem{cor}[subsection]{Corollary}

\theoremstyle{definition}
\newtheorem{remark}[subsection]{Remark}

\newtheorem{example}[subsection]{Example}

\numberwithin{equation}{section}

\newcommand{\DS}{\displaystyle }

\newcommand{\A}{{\mathcal A}}

\newcommand{\La}{{\mathcal L}}

\newcommand{\Q}{{\mathbb Q}}

\newcommand{\C}{{\mathbb C}}

\newcommand{\G}{\mathsf{G}}

\renewcommand{\P}{{\mathbb P}}

\renewcommand{\L}{{\mathbb L}}

\newcommand{\h}{\mathfrak{h}}

\begin{document}

\title[Holonomy Lie algebras and the LCS formula]%
{Holonomy Lie algebras and the LCS formula \\for subarrangements of $A_n$}

\author[Paulo Lima-Filho]{Paulo Lima-Filho}
\address{Department of Mathematics,
Texas A\&M University, College Station, TX 77843}
\email{\href{mailto:plfilho@math.tamu.edu}{plfilho@math.tamu.edu}}
\urladdr{\href{http://www.math.tamu.edu/~paulo.lima-filho/}%
{http://www.math.tamu.edu/\~{}paulo.lima-filho}}

\author[Hal Schenck]{Hal Schenck$^1$}
\address{Department of Mathematics,
University of Illinois, Urbana, IL 61801}
\email{\href{mailto:schenck@math.uiuc.edu}{schenck@math.uiuc.edu}}
\urladdr{\href{http://www.math.uiuc.edu/~schenck/}%
{http://www.math.uiuc.edu/\~{}schenck}}

\thanks{$^1$Partially supported by NSF DMS 07-07667, NSA MDA
  904-03-1-0006.}



\begin{abstract}
If $X$ is the complement of a hypersurface in $\P^n$, then Kohno
showed in \cite{K83} that the nilpotent completion of $\pi_1(X)$ is isomorphic
to the nilpotent completion of the holonomy Lie algebra of $X$.
When $X$ is the complement of a hyperplane arrangement $\A$, the
ranks $\phi_k$ of the lower central series quotients of $\pi_1(X)$
are known for isolated examples, and for two special classes: 
arrangements for which $X$ is hypersolvable
(in which case the quadratic closure of the cohomology ring is Koszul), or if 
the holonomy Lie algebra decomposes in degree $3$ as a 
direct product of local components.
In this paper, we use the holonomy Lie algebra to obtain a formula
for $\phi_k$ when $\A$ is a subarrangement of $A_n$. This extends
Kohno's result \cite{K85} for braid arrangements, and provides
an instance of an LCS formula for arrangements which are not
decomposable or hypersolvable.
\end{abstract}

\maketitle


\section{Introduction}
\label{sec:intro}
Let $X$ be the complement of an arrangement of projective hypersurfaces.
Associated to the fundamental group $\pi_1(X)$ of $X$ is the
{\em lower central series}: a chain of normal subgroups
\[
\pi_1(X)=G_1 \ge G_2 \ge G_3\ge\cdots,
\]
defined inductively by $G_k = [G_{k-1},G_1]$. The graded vector
space $(\oplus G_i/G_{i+1})\otimes \Q$ becomes a graded Lie
algebra over $\Q$ with bracket given by $[x,y]=xyx^{-1}y^{-1}$. In
\cite{K83}, Kohno combined results of Deligne \cite{De} and Morgan \cite{M} on
the bigrading on the mixed Hodge structure
of $H^1(X,\Q)$ with results of Sullivan \cite{S} on the minimal
model to show that $(\oplus G_i/G_{i+1})\otimes \Q$ is 
isomorphic to the nilpotent completion of the holonomy Lie algebra of $X$.

If $X$ is the complement of an arrangement of hyperplanes $\A$,
the holonomy Lie algebra $\h_X$ is a purely combinatorial object,
defined in terms of the rank two flats in the intersection lattice
$\La(\A)$. Elements of rank $i$ in $\La(\A)$ correspond to maximal
sets of hyperplanes $\{H_1, \ldots, H_k \}$ intersecting in
codimension exactly $i$. Slicing an arrangement $\A$ with a
generic plane yields an arrangement $\A' \subseteq \P^2$ with
$\La(\A) \simeq \La(\A')$ in rank $\le 2$. In a similar vein, a
theorem of Hamm-L\^{e} \cite{HL} shows that the fundamental group of $X$ is
unaffected by the process of slicing down generically to an arrangement in
$\P^2$, which shows that the algebra $(\oplus G_i/G_{i+1})\otimes \Q$ 
depends only on a generic planar restriction of $X$. 
Comparing these two reductions suggests the possibility of a 
result like that proved by Kohno. An immediate consequence of
the isomorphism $\h_X \simeq (\oplus G_i/G_{i+1})\otimes \Q$ is
that $\phi_k$ is  the dimension of the $k^{th}$ graded piece
of $\h_X$, where $\h_X$ is graded by bracket depth. 

In \cite{K85}, Kohno used this framework to prove a beautiful formula
for the LCS ranks of the braid arrangement $A_n$:
\begin{equation}
\label{lcsformula}
\prod_{k=1}^{\infty}(1-t^k)^{\phi_k}=P(X,-t) = \prod_{i=1}^{n-1}(1-it),
\end{equation}
where $P(X,t)=\sum b_it^i $ is the Poincar\'e polynomial of the
complement of $X$ in $\C^n$.
\begin{defn}
Let $G$ be a (simple) graph on $\ell$ vertices, with edge-set $\mathsf{E}$,
and let $\A_{G}=\{z_i-z_j=0\mid (i,j)\in \mathsf{E} \}$
be the corresponding arrangement in $\C^{\ell}$; $\A_{G}$ is called
a graphic arrangement. Write $U_{G}(t)$ for $\prod_{k=1}^{\infty}
\left(1-t^k\right)^{\phi_k(\pi_1(\C^{\ell}  \setminus \A_{G}))}$.
\end{defn}
The braid arrangement $A_n$ is the arrangement associated to the
graph $K_n$. Our main result is an extension (conjectured in
\cite{SS}) of Kohno's theorem:
\begin{thm}[Graphic LCS formula]
\label{thm:gLCS}
If $G$ is a graph with $\ell$ vertices, and $\kappa_s$ denotes the
number of complete subgraphs of $G$ on $s+1$ vertices, then:
\begin{equation}
\label{eq:GLCS}
U_{G}(t)=
\prod_{j=1}^{\ell-1} \left(1-jt\right)^{\DS{\sum_{s=j}^{\ell-1}(-1)^{s-j}
\tbinom{s}{j} \kappa_{s}
}}
\end{equation}
\end{thm}
\begin{example}
Suppose $G$ consists of the $1$-skeleton of the Egyptian pyramid and
the one skeleton of a tetrahedron, sharing a single triangle, as
below:
\vskip .15in
\begin{center}
\epsfig{file=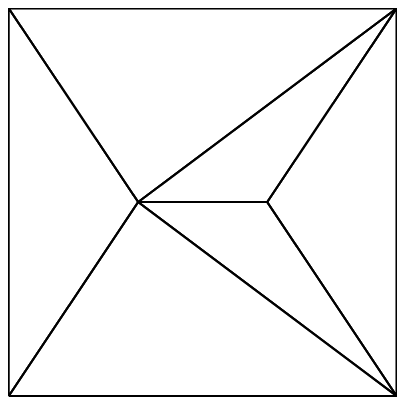,height=1.0in,width=2.0in}
\end{center}
\vskip .15in For this graph, 
$U_{G}(t) =(1-2t)^4(1-3t)$. The proof of $(1.2)$ 
uses a Mayer-Vietoris argument: a graph is determined by the 
data of the maximal complete subgraphs, together with gluing data. 
For example, if $G'$ denotes the $1$-skeleton of the  Egyptian pyramid, 
then four triangles are glued along four edges, and 
Corollary~\ref{cor:poly} shows that 
$U_{G'}(t) =((1-t)(1-2t))^4/(1-t)^4 =(1-2t)^4$. For the $K_4$ subgraph, 
$U_{K_4}(t) =(1-t)(1-2t)(1-3t)$. The graphs $G'$ and $K_4$ are glued 
along a common triangle, hence
\[
U_G(t) = \frac{(1-2t)^4 \cdot (1-t)(1-2t)(1-3t)}{(1-t)(1-2t)} .
\]
\end{example}
\subsection{Other LCS formulas}
To place (\ref{eq:GLCS}) in context, we note that Kohno's formula has
previously been extended in two different directions. In
\cite{FR85}, Falk and Randell give a generalization to fiber-type
arrangements, i.e. those whose complement can be factored as a
chain of {\em linear} fibrations with fiber $F_i \simeq \C -\{
p_1,\ldots, p_{i_k} \}$. As shown by Terao in \cite{T}, this is equivalent
to a lattice-theoretic condition known as {\em supersolvability}.
Supersolvability implies that $H^{*}(X)$ admits a quadratic
Gr\"obner basis, which implies that $H^{*}(X)$ is Koszul; in
\cite{SY}, Shelton and Yuzvinsky interpret the LCS formula
\eqref{lcsformula} in terms of Koszul duality. The requirement
that $H^{*}(X)$ be Koszul can be relaxed a bit, as shown in
\cite{JP}, where Jambu and Papadima provide a version of the LCS
formula for hypersolvable arrangements. An arrangement ${\mathcal A}$ is
hypersolvable iff there is a supersolvable arrangement ${\mathcal B}$ and
linear subspace $W$ with ${\mathcal A} = {\mathcal B} \cap W$ preserving the
rank two flats (see \cite{DS}). Since an arrangement 
in $\P^2$ is supersolvable iff there exists $x \in
\La_2(\A)$ such that for each $y \in \La_2(\A)$ there is a line of
$\A$ connecting $x$ and $y$, it is clear that supersolvable
arrangements have very constrained geometry.

The second generalization lies at the opposite end of the
spectrum. Following Papadima and Suciu \cite{PS}, define
an arrangement to be {\em decomposable} if the
holonomy Lie algebra behaves in degree~$3$ like a direct
product of holonomy Lie algebras of pencil arrangements. 
Arrangements of this type are as simple as
possible in many respects; as shown by Cohen-Suciu \cite{CS}, 
the linearized Alexander invariant is a direct sum of
submodules corresponding to Alexander invariants of rank two
subarrangements, with scalars suitably extended. 
While the underlying arrangement $\A$ is not in general
a product, the LCS ranks behave as if it were. Papadima-Suciu
prove that for decomposable arrangements, 
\begin{equation}
\label{eq:decomp}
U_{G}(t)=(1-t)^{b_1} \prod_{p\in L_2(\A)} \frac{1- \mu(p)\, t}{1-t}.
\end{equation}
Here $\mu$ denotes the M\"obius function: for $p\in L_2(\A)$, 
$\mu(p)$ is one less than the number of hyperplanes of $\A$ which
contain $p$. 
\begin{remark}
Theorem 6.4 of \cite{SS} shows that a graphic arrangement 
has a cohomology ring which is  Koszul iff the graph is supersolvable,
and in \cite{PS2}, Papadima and Suciu give necessary and sufficient
conditions for a graphic arrangement to be hypersolvable.
Results of Stanley \cite{stan} and Terao \cite{T} show 
that $\A_G$ is supersolvable iff $G$ is chordal.
As noted in Remark 6.18 of \cite{SS}, the chromatic 
polynomial of a chordal graph is given by
\begin{equation}
\label{eq:chordchi}
\chi_{\G}(t)=t^{\kappa_0} \prod_{j=1}^{\kappa_0-1}
\left(1-jt^{-1}\right)^{\sum_{s=j}^{\kappa_0-1}(-1)^{s-j}
\tbinom{s}{j} \kappa_{s}},
\end{equation}
and in this case, $(1.2)$ specializes to the
Falk-Randell LCS formula. At the opposite end of the spectrum 
are the decomposable graphic arrangements, 
classified in \cite{PS}: they are exactly 
those with $\kappa_i = 0$ for all $i\ge 3$. For these arrangements, 
(\ref{eq:GLCS}) specializes to (\ref{eq:decomp}). Equations \ref{eq:GLCS}
and \ref{eq:decomp} are instances of the Resonance LCS conjecture of
\cite{Su}. 
\end{remark}
The point of this paper is that for certain special classes of
arrangements, it is possible to find a formula which encompasses
both the decomposable and supersolvable classes, and indeed, everything in 
between. The main tool involved in the proof is a careful analysis 
of the holonomy Lie algebra, which we quickly review. 
It is worth noting that in \cite{P}, Peeva shows that for 
the $D_3$ arrangement, the LCS ranks cannot be obtained as $P(N,-t)$ 
for any standard graded algebra $N$, and in \cite{R}, Roos 
shows that in fact there are arrangements for which 
\[
\prod_{i=1}^\infty (1-t^i)^{\phi_i}
\]
is a transcendental function.

\subsection{The holonomy Lie algebra and formal spaces}
Let $X$ be the complement of a complex hypersurface in $\P^n$. Cup
product gives a map 
\[
H^1(X,\Q) \wedge H^1(X,\Q) \longrightarrow
H^2(X,\Q); 
\]
the dual of this map is the comultiplication map
\[
H_2(X,\Q) \xrightarrow{\cup^t} H_1(X,\Q)\wedge H_1(X,\Q)
\longrightarrow \L(H_1(X,\Q)),
\]
where $\L(H_1(X,\Q))$ is the free Lie algebra on $H_1(X,\Q).$
Following Chen \cite{C}, define the holonomy Lie algebra of $X$ as
the quotient $\L(H_1(X,\Q))/I_X$, where $I_X$ is the ideal
generated by the image of the comultiplication map. 
In the case of arrangements, Kohno shows that the image of 
the comultiplication map is generated by all brackets of the form
\[
[x_j, \sum_{i=1}^k x_i],
\]
where $x_i$ is the generator of $\L(H_1(X,\Q))$ corresponding to
the $i$-th hyperplane $e_i$ and $e_j \in \{e_1,\ldots, e_k \}$ is a
(maximal) set of hyperplanes intersecting in codimension two;
$\{e_1,\ldots, e_k \}$ corresponds to an element of $\La_2(\A)$.
This is also an immediate consequence 
of the description of the cohomology ring of an arrangement
complement given by Orlik-Solomon in \cite{OS}.

The proof of Kohno's main result runs as follows. In \cite{S},
Sullivan introduced the minimal model of a differential
graded algebra (DGA) and associated notion of {\em formality}.
A space $X$ is formal if the minimal model of the DGA of
$\Q$-polynomial forms on $X$ is formal. Sullivan shows that
for a formal space, $(\oplus G_i/G_{i+1})\otimes \Q$ is
isomorphic to the associated graded Lie algebra $\mathcal{L}$
of a filtration $\mathcal{L}_i$ arising in the construction
of the $1$-minimal model.

If $\A$ is defined by the vanishing of $\{\alpha_1,\ldots,\alpha_n\}$, let 
$\mathcal{R} = \C[\frac{d \alpha_1}{\alpha_1}, \ldots,\frac{d \alpha_n}{\alpha_n}]$. 
Brieskorn \cite{Br} shows that the map $\mathcal{R} \rightarrow H^*(X)$ 
sending $\frac{d \alpha_i}{\alpha_i}$ to its class
is an isomorphism, so an arrangement complement 
$X$ is formal and Sullivan's results apply. 
Using work of \cite{De} and \cite{M}, Kohno analyzes
the bigrading on the Hirsch extensions appearing in
the construction of the $1$-minimal model and proves
that there is an isomorphism from the holonomy Lie algebra
to  $\mathcal{L}$. It follows from Poincar\'e-Birkhoff-Witt 
that $\prod_{k=1}^{\infty} (1-t^k)^{-\phi_k}$
is the Hilbert series of the universal enveloping
algebra of the holonomy Lie algebra.

\section{Gluing $\h_X$ from subgraphs}

In this section, we prove an amalgamation lemma, which will allow
us to relate the holonomy Lie algebras of two graphic
arrangements, identified along a common subgraph of a special
type.

We denote by $I_G$ the ideal in $\L(H_1(X,\Q))$ describing the
holonomy Lie algebra for the graphic arrangement $\A_{G}$. In this
case,  the  brackets generating the ideal are as follows: either
they have the form $[x_1,x_2]$ for edges $e_1, e_2$ not contained
in a triangle of $G$,  or they have the form $[x_i,x_1+x_2+x_3],\
i = 1,2,3$,\ if $\ \{ e_1, e_2, e_3 \}\ $  is a triangle in $G$.
Note that in the latter case it suffices to use the brackets
$[x_1,x_2+x_3]$ and $[x_2,x_1+x_3].$

To simplify notation, we write $\h_G$ for $\h_{\A_G}$ and $\L(G)$
for the free Lie algebra on the \emph{edges} of $G$; edges of a
graph will be denoted by $e_j$, with $x_j$ the corresponding
generator in $\L(G)$ or $\h_G$. Let $\Delta_G$ denote set of
triangles in a graph $G$.
\begin{defn}
\label{defn:star}
For a subgraph $K$ of $G$, the pair $(G,K)$ is triangle-complete
(TC) if whenever $\{ e_{i_1},e_{i_2},e_{i_3} \} \in \Delta_G$
and $\{ e_{j_1},e_{j_2} \} \subseteq K$ with $\{ e_{j_1},e_{j_2} \}
\subseteq \{ e_{i_1},e_{i_2},e_{i_3} \}$, then $\{
e_{i_1},e_{i_2},e_{i_3} \} \subseteq K$. So if two edges of a 
triangle of $G$ lie in $K$, the third edge is also in $K$.
\end{defn}

\begin{lem}
\label{lem:glue}
Let  $b = [x_n[x_{n-1},[x_{n-2},[\ldots [x_1,x_0]]]\ldots] \in
  \h_G$ be a bracket of depth $n$, and suppose that $e_i \in {G \setminus K}$ for some $i \in \{0,\ldots,n\}$.
If $(G,K)$ is TC, then $b$ can be written as a sum of brackets of
depth $n$
\[
b= \sum\ [x_{j_n},[x_{j_{n-1}},[x_{j_{n-2}},[\ldots
[x_{j_1},x_{j_0}]]]\ldots]
\]
with all the $e_{j_k}$'s in ${G \setminus K}$.
\end{lem}

\begin{proof}
We induct on the depth of the bracket. Consider $[x_1,x_0]$, with
$x_1 \in {G \setminus K}$ and $x_0 \in K$. If $\{e_1,e_0 \}$ is
not the edge of a triangle in $G$, then $[x_1,x_0]=0$ in $\h_G$,
and we can replace $x_0$ with $x_1$. Otherwise, there exists an
edge $e_j$ with $\{e_1,e_0, e_j \} \in \Delta_G$, so that in
$\h_G$ we have $[x_1,x_0+x_j]=0$, hence $[x_1,x_0]=[x_j,x_1]$.
Since $(G,K)$ is TC, the assumption that $e_1 \not\in K$ implies
that $e_j \not\in K$.

Let $B$ denote a bracket of the form
$[x_{n-1},[x_{n-2},[\ldots [x_1,x_0]]]\ldots]$, and consider
$[x_n,B]$. If $B$ has some element in ${G \setminus K}$, then by induction we can
expand $B$ as a sum of brackets of the desired type. So it suffices
in this case to prove the result for $[x_n,[x_{n-1},B']]$ with
all entries of $[x_{n-1},B']$ in ${G \setminus K}$. But
\begin{equation}
\label{eq:jacobi}
[x_n,[x_{n-1},B']]=-[B',[x_n,x_{n-1}]]-[x_{n-1},[B',x_n]],
\end{equation}
and the result holds by induction. If $B$ is has no element in ${G
\setminus K}$, then $x_n \in {G \setminus K}$, and again the
result follows from the Jacobi identity \eqref{eq:jacobi} and
induction.
\end{proof}

\begin{cor}
\label{cor:surj}
If $(G,K)$ is TC, then there are canonical surjections $\pi \colon
\h_G \rightarrow \h_K$ and $I_G \rightarrow I_K$. Furthermore, the
kernel $K(\pi)$ of $\pi$ is the subalgebra of $\h_G$ generated by
$\{ x_i \mid e_i \in {G \setminus K} \}$.
\end{cor}
\begin{proof}
The map $\mathsf{E}(G) \to \L(K)$, which is the identity on the
edges of $K$ and sends all other edges to zero, extends to a
surjection of free Lie algebras $\L(G) \xrightarrow{\psi} \L(K)$.
Since $K\subset G,$ one can consider $\L(K)$ as a subalgebra of
$\L(G)$, and $\psi$ restricts to the identity on $\L(K)$.  The
induced surjection $\L(G) \rightarrow \h_K$ descends to a
surjection $\pi \colon \h_G \to \h_K$, since the TC property and
Lemma~\ref{lem:glue} guarantee that any generator of $I_G$ either
is sent to zero, or is unchanged and hence  is a generator of
$I_K$. This not only gives the surjectivity of $\psi_{|I_G} \colon
I_G \rightarrow I_K$, but also shows that
\begin{equation}
\label{eq:ideals}
I_G \cap \L(K) \ = \ I_K,
\end{equation}
whenever $(G,K)$ is TC.  Using the snake lemma, one concludes that
$K(\pi)$ is the ideal in $\h_G$ generated by $\langle x_i \mid e_i
\in {G \setminus K} \rangle$. We now proceed to show it is indeed
the subalgebra generated by these elements.
\begin{equation*}
\xymatrix{
 &
I_{G} \ar[d]\ar[r] &
I_{K} \ar[d]  \\
\langle x_i | e_i \in {G \setminus K} \rangle  \ar[r] \ar[d]^{\delta} &
\L(G)  \ar[r]\ar[d] &
\L(K) \ar[d]  \\
K(\pi) \ar[r]  &
\h_{G} \ar[r] & \h_K
}
\end{equation*}
\vspace*{4pt}

By Lemma~\ref{lem:glue} any bracket $\gamma \in K(\pi)$ involving
some $x_i$ with $e_i \in {G \setminus K}$ may be written as a sum
of brackets with all entries coming from ${G \setminus K}$, so it
suffices to show that if $\alpha\neq 0 \in \h_G$  is a sum of
brackets involving only $x_i$'s with $e_i \in K$ then  $\alpha
\not\in K(\pi)$. Indeed, an element $\alpha$ of this form can be
lifted to an element $\widehat{\alpha} \in \L(K) \subset \L(G)$,
and hence $\pi(\alpha) = 0$ if and only if $\widehat{\alpha} \in
I_K = I_G\cap \L(K) \subset I_G.$ This gives $\alpha =0.$
\end{proof}

Let $G_1$ and $G_2$ be two graphs, having a common subgraph $K$.
We construct a new graph $G = G_1 \bigcup_K G_2$ by identifying
$G_1$ and $G_2$ along the image of $K$.

\begin{thm}
\label{thm:ses}
If $G = G_1 \bigcup_K G_2$ and $(G_i,K)$ and $(G,G_i)$ have
property TC for $i \in \{1,2\}$, then:
\begin{description}
\item[a] The diagram
$$
\xymatrix{ \h_G \ar[r]^{\pi_2} \ar[d]_{\pi_1} & \h_{G_2} \ar[d]^{\phi_2}\\
\h_{G_1} \ar[r]_{\phi_1} & \h_K
 }
$$
is a pull-back diagram of Lie algebras.
\item[b] The diagram
$$
\xymatrix{ \h_K \ar[r]^{i_2} \ar[d]_{i_1} & \h_{G_2} \ar[d]^{j_2}\\
\h_{G_1} \ar[r]_{j_1} & \h_G
 }
$$
is a push-out diagram of Lie algebras.
\end{description}
\end{thm}
\begin{proof}
To prove assertion {\textbf a}, it suffices to show that
\begin{equation}
\label{eq:exact}
 0 \longrightarrow \h_G \xrightarrow{\pi_1 \times \pi_2}
\h_{G_1}\times \h_{G_2}
 \xrightarrow{\phi_1-\phi_2} \h_{K}  \longrightarrow 0.
\end{equation}
is an exact sequence of vector spaces. 
It follows from Corollary \ref{cor:surj} that $\phi_1 - \phi_2$ is
surjective. On the other hand, Lemma \ref{lem:glue} shows that any element
$\alpha_i \in \h_{G_i}$, $i=1,2$, can be written as $\alpha_i = g_i + k_i$,
where $g_i$ (respectively $k_i$) is a sum of brackets involving only generators
associated to edges in $G_i - K$ (respectively $K$). In particular, if
$(\alpha_1,\alpha_2) \in \operatorname{Ker}(\phi_1-\phi_2)$ then $k_1=k_2$
and one immediately sees that $(\alpha_1,\alpha_2)$ is in the image of
$\pi_1\times \pi_2$. The exactness of the sequence in the middle is now
clear. Consider the commutative diagram:

\begin{equation*}
\xymatrix{
 &
K(\pi_1) \ar[d]\ar[r] &
K(\phi_2) \ar[d]  \\
K(\pi_2) \ar[r]\ar[d] & \h_G  \ar[r]^{\pi_2}\ar[d]^{\pi_1} &
\h_{G_2} \ar[d]^{\phi_2}  \\
K(\phi_1) \ar[r] &
\h_{G_1} \ar[r]^{\phi_1} &
\h_K
}
\end{equation*}
\vspace*{4pt}

Since both $(G_i,K)$ and $(G,G_i)$ are TC, it follows from
Lemma~\ref{lem:glue} that any bracket in $K(\phi_2)$ may be
written as a bracket involving only the edges in $G_2 \setminus K
= G \setminus G_1$, and any bracket in $K(\phi_1)$ may be written
as a bracket involving only the edges in $G_1 \setminus K = G
\setminus G_2$. Thus, the maps $K(\pi_1) \rightarrow K(\phi_2)$
and $K(\pi_2) \rightarrow K(\phi_1)$ are surjective. 

By Corollary~\ref{cor:surj}, $K(\pi_i)$ is the subalgebra of $\h_{G}$
generated by the variables associated to edges in $G \setminus
G_i$, $i=1, 2$ and, similarly, $K(\phi_i)$ is the subalgebra of
$\h_{G}$ generated by the variables associated to edges in $G_i
\setminus K$, $i=1, 2.$

It follows that the projection $\L(G) \to \h_G$ induces an
isomorphism
$$\L(G\setminus G_1)/ I_G\cap \L(G\setminus G_1) \simeq
K(\pi_1).
$$
The surjection $\L(G) \to \L(G_2)$ restricts to an isomorphism
\[
\L(G\setminus G_1) \xrightarrow{\simeq} \L(G_2\setminus K)
\]
and under this isomorphism one has
\[
\begin{array}{ccc}
K(\pi_1) & \simeq &\L(G\setminus G_1)/ I_G\cap \L(G\setminus G_1)\\
        &  \simeq &\L(G_2\setminus K)/ I_G \cap\L(G_2)\cap \L(G_2\setminus K) \\
         & = & \L(G_2\setminus K)/I_{G_2}\cap \L(G_2\setminus K) \\
         &\simeq& K(\phi_2),
\end{array}
\]
where the equality comes from \eqref{eq:ideals}, since $(G,G_2)$
is TC. It follows that the upper horizontal arrow in the diagram
is an isomorphism, and so is the leftmost vertical arrow. 
In particular $K(\pi_1)\cap K(\pi_2) = 0$, thus showing that the
kernel of $\pi_1\times \pi_2$ is trivial, and this concludes the proof that
the sequence \eqref{eq:exact} is exact, and $({\bf a})$ follows.

To prove assertion {\bf b}, first observe that \eqref{eq:ideals}
implies that the natural maps $i_r \colon \h_K \to \h_{G_r}$ and
$j_r \colon \h_{G_r}\to \h_G$ are injections, for $r=1,2$. Now,
consider the sequence
\[ 0
\longrightarrow \h_K \xrightarrow{i_1 \times i_2} \h_{G_1}\times
\h_{G_2}
 \xrightarrow{j_1-j_2} \h_{G}  \longrightarrow 0.
\]
A similar argument to the one used in previous assertion shows
that this sequence is exact and that the diagram is in fact a
push-out of Lie algebras.
\end{proof}
\begin{cor}
\label{cor:poly}
With assumptions above,
\[
U_G(t) = \frac{U_{G_1}(t) \cdot U_{G_2}(t)}{U_K(t)} .
\]
\end{cor}

\section{Main Theorem}
\label{sec:graphic}
To prove the main theorem, we induct on the number of vertices
$v(G)$ of $G$, the basic case being trivial. If $G$ is a graph on
$\{v_1,\ldots,v_n \}$, then let $G_1$ denote the subgraph of $G$
consisting of all edges of $G$ not containing $v_n$, and $G_2'$
denote the subgraph of $G$ consisting of all edges of $G$
containing $v_n$. Note that $v(G_1)= n-1$.
\begin{lem}\label{lem:induct}
Let $K$ consist of all edges in $G_1$ which connect vertices
$\{a_1,a_2 \}$ such that $ \{ a_1,v_n \}, \{ a_2,v_n \} \in G_2'$, and
define $G_2 = G_2' \cup K$. Then $(G_i,K)$ and $(G,G_i)$ are TC.
\end{lem}
\begin{proof}
If a triangle with vertices $\{a_1,a_2,a_3\}$ of $G$ has two edges
$\{a_1,a_2\},\{a_1,a_3\}$
in $K$, then $\{ a_i,v_n \} \in G_2'$ and hence $\{a_2,a_3\} \in K$.
So $(G,K)$ is TC, which implies $(G_i,K)$ is TC. Clearly $(G,G_1)$ is
TC, since any triangle of $G$ with two edges in $G_1$ does not contain
the vertex $v_n$, so has the third edge in $G_1$. Finally, if a
triangle of $G$ has two edges $\{a_1,a_2\},\{a_1,a_3\}$ in $G_2$,
then there are three possibilities.
\begin{enumerate}
\item If $v_n = a_1$, then $\{a_2,a_3 \} \in K \subseteq G_2$.
\item If $v_n = a_2$, then $\{a_1,a_3 \} \in K$, hence $\{v_n,a_3 \}\in G_2$.
\item If $v_n \not \in \{a_1,a_2,a_3\}$, then $\{v_n,a_2 \}, \{v_n,a_3 \} \in G_2$,
  so  $\{a_2,a_3 \} \in K \subseteq G_2$. 
\end{enumerate}\end{proof}

If $G \simeq K_n$  then the formula holds by \cite{K85} (this can
also be proved directly, using \cite{FR85}, \cite{stan} and \cite{T}), 
so we may assume $G \not\simeq K_n$. Using the notation of
Lemma~\ref{lem:induct} choose $v_n$ so that $v(G_2) = v(G_2') <
v(G)$. By induction, Lemma~\ref{lem:induct} and
Corollary~\ref{cor:poly}, we find that $U_G(t) = $ \vspace*{4pt}
\begin{small}
\[
\frac{
\prod_{j=1}^{v(G_1)-1} \left(1-jt\right)^{\DS{\sum_{s=j}^{v(G_1)-1}(-1)^{s-j}
\tbinom{s}{j} \kappa_{s}(G_1)}} \cdot
\prod_{j=1}^{v(G_2)-1} \left(1-jt\right)^{\DS{\sum_{s=j}^{v(G_2)-1}(-1)^{s-j}
\tbinom{s}{j} \kappa_{s}(G_2)}}
}
{
\prod_{j=1}^{v(K)-1} \left(1-jt\right)^{\DS{\sum_{s=j}^{v(K)-1}(-1)^{s-j}
\tbinom{s}{j} \kappa_{s}(K)}}
}
\]
\end{small}
\vspace*{4pt}

\noindent Observing that for any graph $H,$ one has
$\kappa_s(H)=0$ whenever $s\geq v(H)$, we can write this formula
as
\[
\prod_{j\geq 1}\ \left(1-jt\right)^{ \left[ \DS{\sum_{s\geq
j}(-1)^{s-j} \tbinom{s}{j} \left\{ \kappa_{s}(G_1) +
\kappa_s(G_2)- \kappa_s(K) \right\} } \right] } .
\]
Partition the complete subgraphs of $G$ into two sets: those
which contain $v_n$, and those which do not. A complete
subgraph lying in $G_2$ and not containing $v_n$ corresponds to a
complete subgraph lying in $K$, and so is counted in both
$G_1$ and $G_2$. This concludes the proof of $(\ref{eq:GLCS})$.
\begin{remark}
\label{rmk:decomp}
The ideas here may be extended to the class of Lie algebras where 
an analog of Theorem~\ref{thm:ses} holds, and will appear elsewhere shortly.
\end{remark}

\bibliographystyle{amsalpha}

\end{document}